\documentclass[11pt]{article}

\usepackage{amsmath, amssymb, amsthm, xypic, stmaryrd, graphicx, mathtools}
\usepackage[all]{xy}
\usepackage[dvipsnames]{xcolor}
\usepackage{tikz}
\usetikzlibrary{cd}
\usepackage{boxedminipage}
\usepackage[shortlabels]{enumitem}
\usepackage[colorlinks=true,linkcolor=blue,citecolor=blue]{hyperref}

\DeclareMathOperator{\supp}{supp}

\DeclareMathOperator{\Real}{Re}
\DeclareMathOperator{\Ind}{Ind}

\DeclareMathOperator{\alg}{alg}
\DeclareMathOperator{\red}{red}

\DeclareMathOperator{\Av}{Av}

\DeclareMathOperator{\U}{U}
\DeclareMathOperator{\Spin}{Spin}

\DeclareMathOperator{\restr}{restr}

\newcommand{\beq}[1]{\begin{equation} \label{#1}}
\newcommand{\eeq}{\end{equation}}

\newcommand{\bea}{\begin{eqnarray}}
\newcommand{\eea}{\end{eqnarray}}

\begin{document}

\theoremstyle{plain}
\newtheorem{theorem}{Theorem}[section]
\newtheorem{thm}{Theorem}[section]
\newtheorem{lemma}[theorem]{Lemma}
\newtheorem{proposition}[theorem]{Proposition}
\newtheorem{prop}[theorem]{Proposition}
\newtheorem{corollary}[theorem]{Corollary}
\newtheorem{conjecture}[theorem]{Conjecture}

\theoremstyle{definition}
\newtheorem{definition}[theorem]{Definition}
\newtheorem{defn}[theorem]{Definition}
\newtheorem{example}[theorem]{Example}
\newtheorem{remark}[theorem]{Remark}
\newtheorem{rem}[theorem]{Remark}

\newcommand{\C}{\mathbb{C}}
\newcommand{\R}{\mathbb{R}}
\newcommand{\Z}{\mathbb{Z}}
\newcommand{\N}{\mathbb{N}}
\newcommand{\Q}{\mathbb{Q}}
\newcommand{\ch}{\textnormal{ch}}
\newcommand{\cU}{\mathcal{U}}

\newcommand{\Supp}{{\rm Supp}}

\newcommand{\field}[1]{\mathbb{#1}}
\newcommand{\bZ}{\field{Z}}
\newcommand{\bR}{\field{R}}
\newcommand{\bC}{\field{C}}
\newcommand{\bN}{\field{N}}
\newcommand{\bT}{\field{T}}
\newcommand{\rede}[1]{\textcolor{red}{#1}}
\newcommand{\bk}{\mathbf{k}}
\newcommand{\cB}{{\mathcal{B} }}
\newcommand{\cK}{{\mathcal{K} }}
\newcommand{\cF}{{\mathcal{F} }}
\newcommand{\cO}{{\mathcal{O} }}
\newcommand{\cE}{\mathcal{E}}
\newcommand{\cS}{\mathcal{S}}
\newcommand{\calL}{\mathcal{L}}

\newcommand{\HH}{{\mathcal{H} }}
\newcommand{\tilH}{\widetilde{\HH}}
\newcommand{\HX}{\HH_X}
\newcommand{\Hpi}{\HH_{\pi}}
\newcommand{\HHpi}{\HH \otimes \HH_{\pi}}
\newcommand{\Ltwopi}{L^2_{\pi}(X, \HHpi)}
\newcommand{\norm}[1]{\| #1\|}
\newcommand{\KK}{K\!K}

\newcommand{\D}{D \hspace{-0.27cm }\slash}
\newcommand{\Dsmall}{D \hspace{-0.19cm }\slash}

\newcommand{\mybigwedge}{\textstyle{\bigwedge}}

\newcommand{\CGmax}{C^*_{G, \max}}
\newcommand{\DGmax}{D^*_{G, \max}}
\newcommand{\CGred}{C^*_{G, \red}}
\newcommand{\CGalg}{C^*_{G, \alg}}
\newcommand{\DGalg}{D^*_{G, \alg}}
\newcommand{\CGker}{C^*_{G, \ker}}
\newcommand{\Cmax}{C^*_{\max}}
\newcommand{\Dmax}{D^*_{\max}}
\newcommand{\Cred}{C^*_{\red}}
\newcommand{\Calg}{C^*_{\alg}}
\newcommand{\Dalg}{D^*_{\alg}}
\newcommand{\Cker}{C^*_{\ker}}
\newcommand{\tilCalg}{\widetilde{C}^*_{\alg}}
\newcommand{\Cpiker}{C^*_{\pi, \ker}}
\newcommand{\Cpialg}{C^*_{\pi, \alg}}
\newcommand{\Cpimax}{C^*_{\pi, \max}}

\newcommand{\Avpi}{\Av^{\pi}}

\newcommand{\Gtc}{\Gamma^{\infty}_{tc}}
\newcommand{\tilD}{\widetilde{D}}
\newcommand{\XoneH}{X^{\HH}_1}

\newcommand{\XX}{\mathfrak{X}}

\def\kt{\mathfrak{t}}
\def\kk{\mathfrak{k}}
\def\kp{\mathfrak{p}}
\def\kg{\mathfrak{g}}
\def\kh{\mathfrak{h}}

\newcommand{\pilamrho}{[\pi_{\lambda+\rho}]}

\newcommand{\Trestr}{\mathcal{T}_{\restr}}
\newcommand{\TdN}{\mathcal{T}_{d_N}}

\newcommand{\omG}{\om/\hspace{-1mm}/G}
\newcommand{\om}{\omega} \newcommand{\Om}{\Omega}
\newcommand{\Spinc}{\Spin^c}

\def\kt{\mathfrak{t}}
\def\kk{\mathfrak{k}}
\def\kp{\mathfrak{p}}
\def\kg{\mathfrak{g}}
\def\kh{\mathfrak{h}}

\newcommand{\ddt}{\left. \frac{d}{dt}\right|_{t=0}}

\newenvironment{proofof}[1]
{\noindent \emph{Proof of #1.}}{\hfill $\square$}

   \newenvironment{dedication}
        {\begin{quotation}\begin{center}\begin{em}}
        {\par\end{em}\end{center}\end{quotation}}
        
\newcommand{\Todo}{\textbf{To do}}

\title{Covering complexity, scalar curvature, and quantitative $K$-theory}
\date{}
\author{Hao Guo$^1$ and Guoliang Yu$^2$\footnote{The second author is partially supported by NSF 1700021, NSF 2000082, and the Simons Fellows Program.}}
\date{%
    $^1$Tsinghua University\\%
    $^2$Texas A\&M University\\[2ex]%
}


\maketitle
\begin{dedication}
{Dedicated to Professor Blaine Lawson on the occasion of his 80th birthday.}
\end{dedication}
\vspace{0.3cm}

\begin{abstract}
We establish a relationship between a certain notion of covering complexity of a Riemannian spin manifold and positive lower bounds on its scalar curvature. This makes use of a pairing between quantitative operator $K$-theory and Lipschitz topological $K$-theory, combined with an earlier vanishing theorem for the quantitative higher index.
\end{abstract}

\tableofcontents
\section{Introduction}
In \cite{GXY3}, a notion of quantitative higher index was introduced for elliptic differential operators on Riemannian manifolds. It is a natural analogue of the higher index and takes values in quantitative operator $K$-theory, a framework that refines operator $K$-theory by taking into account geometric structures of $C^*$-algebras. We refer to \cite{Guentner,Oyono2,Yu1} for more detailed treatments of quantitative operator $K$-theory. 

One of the results proved in \cite{GXY3} was that, in the spin setting, non-vanishing of the quantitative higher index is an obstruction to scalar curvature being bounded below by a positive constant that depends (inversely) on the propagation parameter of the index. The goal of this paper is two-fold. First, we give an explicit pairing between quantitative operator $K$-theory and Lipschitz topological $K$-theory (Theorem \ref{thm pairing defined} and Corollary \ref{cor useful pairing}). This provides one possible avenue for detecting non-vanishing of the quantitative higher index. Second, we apply this pairing to establish a relationship between a certain type of covering complexity of Riemannian spin manifolds and positive lower bounds on scalar curvature. 

Recall that an open cover $\mathcal{U}=\{U_i\}_{i\in I}$ is called good if any finite intersection of its members is contractible. For any $h>0$, the $h$-multiplicity of $\cU$ is the largest $m$ for which there exists an $x\in M$ such that the open $h$-ball around $x$ meets $m$ members of $\cU$. The Lebesgue number of $\cU$ is the largest $\lambda$ such that for each $x\in M$, the open $\lambda$-ball around $x$ is contained within a member of $\cU$.

We prove:

\begin{theorem}
\label{thm main}
For any $R>0$ and positive integer $m$, there exists a positive constant $k(R,m)$ such that the following holds. If $(M,g)$ is a complete Riemannian spin manifold that admits a uniformly bounded good open cover with Lebesgue number $R$ and $h$-multiplicity $m$ for some $h>0$, then 
$$\inf_{x\in M}\kappa_g(x)\leq k(R,m),$$
where $\kappa_g$ is the scalar curvature of $g$.
\end{theorem}
\begin{remark}
We will give a more precise expression for the constant $k(R,m)$ in the proof of Theorem \ref{thm main}: it is an increasing function of $m$ and is proportional to $R^{-2}$.
\end{remark}
The contents of the paper are as follows.
We begin in section \ref{sec Lipschitz} by recollecting some facts about Lipschitz $K$-theory and a Lipschitz homotopy equivalence, to be used in the remainder of the paper. In section \ref{sec closed}, we illustrate the geometric ideas involved in the proof of Theorem \ref{thm main} by first establishing a weaker result, Theorem \ref{thm weaker}, using classical tools. The quantitative refinement of this proof is then the subject of rest of the paper: section \ref{sec quantitative} concerns the explicit pairing between quantitative $K$-theory and Lipschitz $K$-theory, while in section \ref{sec proof general} we apply this pairing to prove Theorem \ref{thm main}.

\section{Lipschitz representatives of $K$-theory}
\label{sec Lipschitz}

In this section we collect together some useful facts about Lipschitz $K$-theory and a Lipschitz homotopy equivalence that will be used in the rest of the paper.

Let $C_0(X)$ be the $C^*$-algebra of $\C$-valued functions vanishing at infinity on a metric space $X$ and $B_r(S)$ be the open $r$-ball around a subset $S\subseteq X$.

For the proof of the main theorem, we will need the existence of certain Lipschitz representatives of $K$-theory elements for finite-dimensional simplicial complexes, established in \cite{WXYdecay}. 
\begin{definition}
\label{def Lipschitz projection}
An element $f\in M_n(C_0(M)^+)$ is said to be \emph{$L$-Lipschitz} if
\begin{equation}
\label{eq Lipschitz}
\norm{f(x)-f(y)}_{M_n(\C)}\leq L\cdot d_M(x,y)
\end{equation}
for all $x,y\in M$, where on the left-hand side we have the operator norm on $M_n(\C)$.	
\end{definition}

\begin{theorem}[{\cite[Theorem 1.5]{WXYdecay}}]
\label{thm Lipschitz}
For each natural number $k$, there exists a constant $L_k$ such that the following holds. If $X$ is a locally compact $k$-dimensional simplicial complex equipped with the simplicial metric, then every class in $K_0(C_0(X))$ can be represented by a formal difference of $L_k$-Lipschitz projections. Further, any class represented by a projection that is constant outside a compact subset $K$ can represented by one that is $L_k$-Lipschitz and constant outside $B_1(K)$.
\end{theorem}
\begin{remark}
\label{rem Lk}
In fact, \cite[Theorem 1.5]{WXYdecay} also provides the analogous result for $K_1(C_0(X))$, but we will not need it in this paper.

The constant $L_k$ can be taken to be of the form $Ck^{C'k}$ for some constants $C$ and $C'$ \cite{WXYdecay}.
\end{remark}


Let $M$ be a Riemannian manifold with a good cover $\cU$. There is a standard way to construct a homotopy equivalence between $M$ and the nerve complex of $\cU$, denoted by $N(\cU)$. Recall that $N(\cU)$ is the simplicial complex whose vertices are the elements of $\mathcal{U}$, and a simplex is filled whenever its vertices have non-empty intersection in $M$. For any partition of unity $\{\varphi_i\}_{i\in I}$ subordinate to $\mathcal{U}$, the map
\begin{align}
\label{eq phi}
\phi\colon M&\to N(\cU),\nonumber\\
x&\mapsto\sum_{i\in I}\varphi_i(x)U_i
\end{align}
is a homotopy equivalence \cite[section 4.G]{Hatcher}.

Now suppose in addition that $\mathcal{U}$ is uniformly bounded with Lebesgue number $R>0$ and $r$-multiplicity $m$ for some $r>0$, as in the hypothesis of Theorem \ref{thm main}. Then $N(\mathcal{U})$ is a locally compact $(m-1)$-dimensional simplicial complex, and $\phi$ is a proper map. We may refine our choice of $\{\varphi_i\}_{i\in I}$ as follows to gain Lipschitz control over $\phi$, where distance on $N(\cU)$ is defined using the simplicial metric $d_{N(\cU)}$. (Recall that for a given simplex $[U_{i_0},\ldots,U_{i_k}]$, $k\leq m-1$, and points $x=\sum_{j=0}^k s_j U_{i_j}$ and $y=\sum_{j=0}^k t_j U_{i_j}$ written as convex combinations of the vertices, 
$$d_{N(\cU)}(x,y)=\sum_{j=1}^k|s_j-t_j|,$$
while in general $d_{N(\cU)}(x,y)$ is defined to be the distance of the shortest path connecting $x$ to $y$, and infinity if no such path exists.) For each $i\in I$, define $\psi_i\colon U_i\to\R$ by $\psi_i(x)=d_M(x,M\backslash U_i)$, where $d_M$ is the Riemannian distance on $M$. Define $\varphi_i$ by
\begin{align}
\label{eq varphi}
\varphi_i(x)&=\frac{\psi_i(x)}{\sum_{j\in I}\psi_j(x)}.
\end{align}
The following lemma is adapted from \cite[Proposition 5.3]{WXYdecay}.
\begin{lemma}
\label{lem Lipschitz phi}
	The map $\phi\colon M\to N(\cU)$ defined using the partition of unity \eqref{eq varphi} is $mR^{-1}$-Lipschitz.
\end{lemma}
\begin{proof}
	Unwrapping the definition of the simplicial metric, it suffices to prove that for any $x,y\in M$ and any simplex $[U_{i_0},\ldots, U_{i_k}]$ containing $\phi(x)$ and $\phi(y)$, where $k\leq m-1$, we have
	$$d_{N(\cU)}(\phi(x),\phi(y))\leq\frac{m}{R}d_M(x,y).$$
	For this it suffices to show that each of the $k+1\leq m$ coordinate functions
	$$x\mapsto\frac{d_M(x,M\backslash U_{i_j})}{\sum_{j=0}^k d_M(x,M\backslash U_{i_j})}$$
	is $\frac{1}{R}$-Lipschitz. By definition of $\phi$, the set $\{U_{i_0},\ldots,U_{i_k}\}$ contains all of the $U_i$ that contain $x$, hence the denominator is bounded below by the Lebesgue number $R$. We conclude by observing that the numerator is $1$-Lipschitz.
\end{proof}



\section{A weaker result for closed manifolds}
\label{sec closed}
We now wish to illustrate the geometric ideas that underlie the proof of Theorem \ref{thm main} by proving a weaker version of that theorem. As this section serves only an expository role in relation to the rest of the paper, it may be skipped without loss of continuity.

We consider the case where $M$ is closed and the constant $k(m,R)$ is allowed to depend on the dimension of $M$. The argument uses relatively classical computations involving Dirac operators, and an excellent reference is \cite[section II.8]{Lawson-Michelsohn}. We prove:
\begin{theorem}
\label{thm weaker}
For any $R>0$ and positive integers $m$ and $l$, there exists a constant $k(R,m,l)$ such that the following holds. If $(M,g)$ is a closed $l$-dimensional Riemannian spin manifold that admits a uniformly bounded good open cover with Lebesgue number $R$ and $R$-multiplicity $m$, then 
$$\inf_{x\in M}\kappa_g(x)\leq k(R,m,l),$$
where $\kappa_g$ is the scalar curvature of $g$.
\end{theorem}
The proof of Theorem \ref{thm weaker} uses the following lemma, which allows a continuous Lipschitz projection to be replaced by a smooth Lipschitz projection without changing its $K$-theory class.
\begin{lemma}
\label{lem smooth projection}
Let $p$ be an $L$-Lipschitz projection in $M_n(C(M))$ in the sense of Definition \ref{def Lipschitz projection} for some $L>0$ and positive integer $n$. For any $\lambda>0$, there exists a $(2L+\lambda)$-Lipschitz projection $p'\in M_n(C^\infty(M))$ such that $[p]=[p']\in K_0(C(M))$.
\end{lemma}
\begin{proof}
	First, for any $\varepsilon>0$, we can find a self-adjoint $f_\varepsilon\in M_n(C^\infty(M))$ which is $L$-Lipschitz and such that $\norm{f_\varepsilon-p}\leq\varepsilon$. We see that
	\begin{align*}
	\norm{f_\varepsilon^2-f_\varepsilon}&=\norm{f_\varepsilon^2+(p-f_\varepsilon)+(p^2-p)-p^2}\\
	&\leq\norm{f_\varepsilon^2-p^2}+\varepsilon\\
	&\leq\norm{f_\varepsilon-p}\norm{f_\varepsilon+p}+\varepsilon\\
	&\leq\varepsilon(3+\varepsilon),
	\end{align*}
which we may assume is at most $1/16$. Letting $\delta_\varepsilon=\sqrt{\varepsilon(3+\varepsilon)}$, one sees that the spectrum of $f_\varepsilon$ is contained in the disjoint union of the open balls of radii $\sqrt{\delta_\varepsilon}$ around $0$ and $1$. Let $\Theta$ be the characteristic function of the set $\{z\in\C\colon\Real z\geq 3/8\}\subseteq\C$. Let $E$ be the closed disk of radius $1/2$ and center $1$, and let $\Gamma=\partial E$, positively oriented. Then $\Theta$ is holomorphic on an open neighborhood of $E$, and we may apply the holomorphic functional calculus to obtain a smooth projection
\begin{equation}
\label{eq Cauchy}	
\Theta(f_\varepsilon)=\frac{1}{2\pi i}\int_\Gamma(\xi-f_\varepsilon)^{-1}\,d\xi\in M_n(C^\infty(M)).
\end{equation}
In order to estimate the Lipschitz constant for $\Theta(f_\varepsilon)$, observe that
$$\inf_{\xi\in\Gamma}\rho(\xi-f_\varepsilon)\geq\frac{1}{2}-\delta_\varepsilon,$$
where $\rho$ denotes the spectral radius. Hence $\norm{(\xi-f_\varepsilon)^{-1}}\leq 2(1-2\delta_\varepsilon)^{-1}$, and for all $x,y\in M$, the norm of $(\xi-f_\varepsilon(x))^{-1}-(\xi-f_\varepsilon(y))^{-1}$ in $M_n(\C)$ is bounded by
\begin{align*}
	L\norm{(\xi-f_\varepsilon)^{-1}}^2\cdot d_M(x,y)&\leq\frac{4L}{(1-2\delta_\varepsilon)^2}\cdot d_M(x,y),
\end{align*}
where 
$d_M$ is the Riemannian distance on $M$.
Thus $(\xi-f_\varepsilon)^{-1}$ is $4L(1-2\delta_\varepsilon)^{-2}$-Lipschitz. It follows from \eqref{eq Cauchy} that $\Theta(f_\varepsilon)$ is $2L(1-2\delta_\varepsilon)^{-2}$-Lipschitz. Thus for any $\lambda>0$, there exists an $\varepsilon$ such that $\Theta(f_\varepsilon)$ is a $(2L+\lambda)$-Lipschitz projection in $M_n(C^\infty(M))$. 

Now $\lim_{\varepsilon\to 0}\norm{f_\varepsilon-\Theta(f_\varepsilon)}=\lim_{\varepsilon\to 0}\sqrt{\delta_\varepsilon}=0.$ As $\norm{f_\varepsilon-p}\leq\varepsilon$, this implies that for $\varepsilon$ small enough, $p'\coloneqq\Theta(f_\varepsilon)$ and $p$ define the same class in $K_0(C(M))$. 
\end{proof}

\begin{proofof}{Theorem \ref{thm weaker}}
We may suppose that $l$ is even; the odd case follows by applying the result in the even case to $M\times S^1$ with the product metric. Let $D$ be the Dirac operator on the spinor bundle $S\to M$, and let $\nabla_S$ denote the spinor connection. We first claim that there exists a smooth Hermitian vector bundle $(E,\nabla_E)$ over $M$ such that:
\begin{itemize}[noitemsep]
	\item $E$ is the pull-back of of a vector bundle $F$ over $N(\cU)$;
	\item the twisted Dirac operator $D_E=D\otimes\nabla_E$ has non-vanishing Fredholm index. 
\end{itemize}
Before justifying this claim, let us fix some notation. Let $\Phi^*\colon H^*(N(\cU))\to H^*(M)$ be the isomorphism on cohomology induced functorially by the map $\phi$ from \eqref{eq phi}. Let
\begin{align*}
\ch_M\colon K^0(M)\otimes\Q&\to H^{\textnormal{even}}(M;\Q),\\
\ch_N\colon K^0(N(\cU))\otimes\Q&\to H^{\textnormal{even}}(N(\cU);\Q),
\end{align*}
be the isomorphisms induced by the Chern character. Let $\langle\cdot,\cdot\rangle_H$ denote the natural pairing between cohomology and homology and $[M]\in H_l(M;\Q)$ the rational fundamental class determined by the spin structure on $M$.

For the claim, note that there exists some $x\in H^l(M)$ such that 
$$\big\langle x,[M]\big\rangle_H\neq 0.$$ 
Let $y=(\Phi^*)^{-1}(x)\in H^l(N(\cU))$. Then there exist $a\in K^0(N(\cU))$ and an integer $c$ such that $\ch_N(a\otimes 1)=(\Phi^*)^{-1}(cx)$.
Thus
\begin{align*}
\big\langle\widehat{A}(M)\cup
\ch_M(\phi^*a),[M]\big\rangle_H&=\big\langle\widehat{A}(M)\cup\Phi^*(\ch_N(a)),[M]\big\rangle_H\\
&=\big\langle\widehat{A}(M)\cup\Phi^*(cy),[M]\big\rangle_H\\
&=\big\langle\widehat{A}(M)\cup cx,[M]\big\rangle_H\\
&=c\big\langle x,[M]\big\rangle_H\neq 0.
\end{align*}
Finally, write $a$ as $[F_0]-[F_1]$ for some vector bundles $F_0$ and $F_1$ over $N(\cU)$. Without loss of generality, we may then take $E=\phi^*F_0$, which can be assumed to be smooth, equipped with any Hermitian connection $\nabla_E$. It then follows from the Atiyah-Singer index theorem that the Fredholm index of $D_E$ is non-zero.

Having established this, let $\mathcal{U}=\{U_i\}_{i\in I}$ be a good cover of $M$ with Lebesgue number $R$ and $R$-multiplicity $m$, which exists by hypothesis. Let $\phi\colon M\to N(\cU)$ be the homotopy equivalence \eqref{eq phi}. By Theorem \ref{thm Lipschitz}, we may assume that the bundle $F_0$ is given by a projection that is $L_{(m-1)}$-Lipschitz in the sense of Definition \ref{def Lipschitz K-theory}. By Lemma \ref{lem Lipschitz phi}, the map $\phi$ is $mR^{-1}$-Lipschitz, so the bundle $E=\phi^*F_0$ is given by an $mR^{-1}L_{m-1}$-Lipschitz projection $p_E\in M_n(C(M))$ for some $n$. Further, by Lemma \ref{lem smooth projection} below, we may assume that $p_E$ is a smooth $(2mR^{-1}L_{m-1}+\lambda)$-Lipschitz projection for some arbitrarily small $\lambda>0$.

Endow the trivial bundle $M\times \C^n$ with the connection $d_n=d\otimes I_n$, and $E$ with the connection $\nabla_E=p_E d_np_E$. Then the curvature of $E$ can be written as
$$R^E=(d_np_E)(I_n-p_E)(d_np_E),$$
and it follows from the above discussion that 
\begin{equation}
\label{eq curvature bound}
\norm{R^E}\leq(2mR^{-1}L_{m-1}+\lambda)^2.	
\end{equation}
The Bochner-Lichnerowicz formula \cite[p. 164]{Lawson-Michelsohn} for $D_E^2$ states that
\begin{equation}
\label{eq D squared}
D_E^2=\nabla^*_E\nabla_E+\frac{\kappa_g}{4}+\mathfrak{R}^E,
\end{equation}
where $\mathfrak{R}^E$ is an endomorphism of $S\otimes E$ that locally takes the form
$$\frac{1}{2}\sum_{i,j=1}^{l}c(e_i)c(e_j)R^E(e_i,e_j).$$
Here, $c(e_i)$ denotes Clifford multiplication by an element $e_i$ of an orthonormal tangent frame. It follows that the operator norm of $\mathfrak{R}^E$ is bounded by $\frac{l^2}{2}\norm{R^E}$. Combining this with \eqref{eq curvature bound} and \eqref{eq D squared} gives
\begin{align}
\label{eq Lichnerowicz bound}
D_E^2&\geq\frac{\kappa_g}{4}-2l^2(mR^{-1}L_{m-1}+\tfrac{\lambda}{2})^2.
\end{align}
Define $k(R,m,l)=8l^2(mR^{-1}L_{m-1})^2$, and suppose $\kappa_g(x)>k(R,m,l)$ for all $x\in M$. Since $\lambda$ can be taken to be arbitrarily small, \eqref{eq Lichnerowicz bound} implies $D_E$ is invertible, which is a contradiction as $D_E$ has non-vanishing index by construction.
\end{proofof}
\begin{remark}
The above argument can be adapted to the non-compact case by means of Gromov-Lawson's relative index theorem \cite{Gromov-Lawson2}.
\end{remark}


\section{A pairing between quantitative and Lipschitz $K$-theories}
\label{sec quantitative}
\noindent The proof of Theorem \ref{thm weaker} in the closed case given in section \ref{sec closed} essentially used a pairing between the Dirac operator $D$ with certain Lipschitz vector bundles over $M$ to detect non-vanishing of the index. The proof of Theorem \ref{thm main} involves an analogous argument carried out using quantitative $K$-theory. More specifically, we will detect non-vanishing of the \emph{quantitative higher index} of $D$ by pairing it with elements from the Lipschitz-controlled topological $K$-theory of $M$.


\subsection{Preliminaries}

\noindent We first recall the set-up of quantitative $K$-theory. It is helpful to start with the general notion of a geometric $C^*$-algebra \cite{Oyono2}.
\begin{definition}
\label{def geometric algebra}
A $C^*$-algebra $A$ is said to be \emph{geometric} if it admits a filtration $\{A_r\}_{r>0}$, where each $A_r$ is a closed linear subspace, satisfying:
\begin{enumerate}[(i)]
\item $A_r\subseteq A_{r'}$ if $r\leq r'$;
\item $A_r A_{r'}\subseteq A_{r+r'}$;
\item $\displaystyle\bigcup_{r>0} A_r$ is dense in $A$.
\end{enumerate}
An element $a\in A_r$ is said to have \emph{propagation at most $r$}, and we write
$$\textnormal{prop}(a)\leq r.$$
\end{definition}
If $A$ is a geometric $C^*$-algebra, each matrix algebra $M_n(A)$ is a geometric $C^*$-algebra with filtration $\{M_n(A_r)\}_{r>0}.$ If $A$ is non-unital, then its unitization $A^+$ is a geometric $C^*$-algebra with filtration $\{A_r\oplus\mathbb{C}\}_{r>0}.$ 

The particular example of a geometric $C^*$-algebra that we will be interested in is the Roe algebra, which we now recall. Let $(M,g)$ be a Riemannian spin manifold of positive dimension, $S\to M$ the spinor bundle, and $d_M$ the Riemannian distance.
	
	\begin{definition}
	\label{def:suppprop}
		Given a bounded operator $T$ on $L^2(S)$, we say that:
		\begin{itemize}
			\item The \emph{support} of $T$, denoted $\textnormal{supp}(T)$, is the complement of all $(x,y)\in X\times X$ for which there exist $f_1,f_2\in C_0(X)$ such that $f_1(x)\neq 0$, $f_2(y)\neq 0$, and
			$$f_1Tf_2=0;$$
			\item The \textit{propagation} of $T$ is the extended real number $$\textnormal{prop}(T)=\sup\{d_M(x,y)\,|\,(x,y)\in\textnormal{supp}(T)\};$$
			\item $T$ is \textit{locally compact} if $fT$ and $Tf\in\mathcal{K}(L^2(S))$ for all $f\in C_0(M)$.
		\end{itemize}
		The \emph{algebraic Roe algebra}, denoted by $\C[M]$, is the $*$-subalgebra of $\mathcal{B}(\mathcal{H})$ consisting of locally compact operators with finite propagation. The \emph{Roe algebra} of $M$, denoted by $C^*(M)$, is the completion of $\C[M]$ with respect to the operator norm. 
	\end{definition}
	The Roe algebra $C^*(M)$ is a geometric $C^*$-algebra in the sense of Definition \ref{def geometric algebra} with respect to the filtration $\{\mathbb{C}[M]_r\}_{r>0}$, where
\begin{equation}
\label{eq filtration}
\mathbb{C}[M]_r\coloneqq\{T\in\mathbb{C}[M]\colon\textnormal{prop}(T)\leq r\}.
\end{equation}


The geometric structure of a $C^*$-algebra allows one to define its \emph{quantitative $K$-theory} groups, which are refinements of the usual operator $K$-theory groups. For this paper, it will suffice to limit ourselves to the even-degree case. We will find it convenient to use the picture of quantitative $K$-theory in terms of quasiidempotents instead of quasiprojections, as done in \cite{Chung}.
\begin{definition}
\label{def quantitative K}
Let $A$ be a unital geometric $C^*$-algebra. Suppose we have real numbers $\varepsilon,r>0$, and $N\geq 1$. An element $e\in A$ is called an \emph{$(\varepsilon,r,N)$-quasiidempotent} if
$$\norm{e^2-e}<\varepsilon,\qquad e\in A_r,\qquad\max(\norm{e},\norm{1_{A}-e})\leq N.$$
Suppose now $0<\varepsilon<\frac{1}{20}$. Let $\textnormal{Idem}_n^{\varepsilon,r,N}(A)$ denote the subspace of $(\varepsilon,r,N)$-quasiidempotents in $M_n(A)$. We have for each positive integer $n$ an injection 
$$\textnormal{Idem}^{\varepsilon,r,N}_n(A)\hookrightarrow\textnormal{Idem}^{\varepsilon,r,N}_{n+1}(A)$$ given by inclusion into the upper-left corner, with respect to which we define 
$$\textnormal{Idem}_\infty^{\varepsilon,r,N}(A)=\varinjlim \textnormal{Idem}^{\varepsilon,r,N}_n(A),$$
where the direct limit is taken in the category of topological spaces.
Define an equivalence relation $\sim$ on $\textnormal{Idem}_\infty^{\varepsilon,r,N}(A)$ by declaring $e\sim f$ if $e$ and $f$ are homotopic through $(4\varepsilon,r,4N)$-quasiidempotents. Denote the equivalence class of an element $e\in\textnormal{Idem}_\infty^{\varepsilon,r,N}(A)$ by $[e]$. Define
$$[e]+[f]=\left[\begin{pmatrix}e&0\\0&f\end{pmatrix}\right].$$
Under this operation, $\textnormal{Idem}_\infty^{\varepsilon,r,N}(A)/\sim$ becomes an abelian monoid with identity $[0]$. Let $K_0^{\varepsilon,r,N}(A)$ be its Grothendieck completion.
\end{definition}
\begin{remark}
\label{rem nonunital}
If $A$ is non-unital, then we have a homomorphism
$$\pi_*\colon K_0^{\varepsilon,r,N}(A^+)\to K_0^{\varepsilon,r,N}(\mathbb{C})$$
induced by the canonical $*$-homomorphism $\pi\colon A^+\to\mathbb{C}$. In this case, we define $K_0^{\varepsilon,r,N}(A)=\ker(\pi_*)$.
\end{remark}

Quantitative operator $K$-theory can be related to the usual operator $K$-theory as follows. If $e$ is an $(\varepsilon,r,N)$-quasiidempotent in a unital $C^*$-algebra $A$ with $\varepsilon<\frac{1}{4}$, then the spectrum of $e$ is contained in the union of disjoint balls $B_{\sqrt{\varepsilon}}(0)\cup B_{\sqrt{\varepsilon}}(1)\subseteq\mathbb{C}$. Choose a function $f_0$ that is holomorphic on a neighborhood of the spectrum and such that
$$
f_0(z)\equiv
\begin{cases}
0&\textnormal{ if }z\in\overline{B}_{\sqrt{\varepsilon}}(0),\\
1&\textnormal{ if }z\in\overline{B}_{\sqrt{\varepsilon}}(1).
\end{cases}
$$
Let $\gamma$ be the contour
$$\{z\in\mathbb{C}\colon|z|=\sqrt{\varepsilon}\}\cup\{z\in\mathbb{C}\colon|z-1|=\sqrt{\varepsilon}\}$$
containing the spectrum of $e$.
Applying the holomorphic functional calculus, we obtain an idempotent
\begin{equation*}
\label{eq f0}
f_0(e)=\frac{1}{2\pi i}\int_{\gamma}f_0(z)(z-e)^{-1}\,dz.
\end{equation*}
(This procedure applies similarly to $(\varepsilon,r,N)$-quasiidempotents in matrix algebras.) The assignment $e\mapsto f_0(e)$ induces a group homomorphism
\begin{align}
\label{eq kappa}
\kappa\colon K_0^{\varepsilon,r,N}(A)\to K_0(A).
\end{align}
(See \cite[Proposition 3.19]{Chung}.) 

Next, let us review \emph{Lipschitz controlled $K$-theory}, developed in \cite{WXYdecay}, in the special case that we will need.

\begin{definition}[{\cite[section 4]{WXYdecay}}]
\label{def Lipschitz K-theory}
Let $M$ be a Riemannian manifold with distance $d_M$. For any $L>0$ and positive integer $n$, let $P_n^L(C_0(M)^+)$ denote the subspace of $L$-Lipschitz projections in $M_n(C_0(M)^+)$ in the sense of Definition \ref{def Lipschitz projection}.

We have for each $n$ an injection
$$P_n^L(C_0(M)^+)\hookrightarrow P_{n+1}^L(C_0(M)^+)$$
given by inclusion into the upper-left corner, with respect to which we define
$$P^L_\infty(C_0(M)^+)=\varinjlim P^L_n(C_0(M)^+),$$
where the direct limit is taken in the category of topological spaces.
Define an equivalence relation $\sim$ on $P^L_\infty(C_0(M)^+)$ by declaring $p\sim q$ if $p$ and $q$ are homotopic through elements in $P_\infty^{2L}(C_0(M)^+)$.

 Denote the equivalence class of an element $p\in P^L_\infty(C_0(M)^+)$ by $[p]$. Define
$$[p]+[q]=\left[\begin{pmatrix}p&0\\0&q\end{pmatrix}\right].$$
Under this operation, $P^L_\infty(C_0(M)^+)/\sim$ becomes an abelian monoid with identity $[0]$. Let $K_0^L(C_0(M)^+)$ be its Grothendieck completion.

The canonical $*$-homomorphism $\pi\colon C_0(M)^+\to\mathbb{C}$ induces a homomorphism $\pi_*\colon K_0^L(C_0(M)^+)\to\Z$. Define $K_0^L(C_0(M))=\ker(\pi_*)$.

\end{definition}

Finally, we review the notion of the quantitative higher index. This concept was introduced in \cite{GXY3}, inspired by questions of Gromov on the geometric size of scalar curvature \cite{GromovLectures}. In the spin setting, it captures the interplay between scalar curvature on the one hand and propagation of index-theoretic information on the other.

Let $M$ be an even-dimensional Riemannian spin manifold with Dirac operator $D$. 
	
	\begin{definition}
	\label{def qhi even}	
	Fix $0<\varepsilon<\tfrac{1}{20}$, $r>0$, and $N\geq 7$. Let $\chi$ be a normalising function, i.e. a continuous odd function $\R\to\R$ such that $\lim_{x\to\infty}\chi(x)=1$, satisfying
	\begin{equation}
	\supp\widehat{\chi}\subseteq\left[-\tfrac{r}{5},\tfrac{r}{5}\right]\quad\textnormal{and}\quad\norm{\chi}_\infty\leq 1,
	\end{equation}
	where $\widehat{\chi}$ is the distributional Fourier transform of $\chi$. With respect to the $\Z_2$-grading on $S=S^+\oplus S^-$, we can write
$$D=\begin{pmatrix}0&D^-\\D^+&0\end{pmatrix},\quad\chi(D)=\begin{pmatrix}0&V\\U&0\end{pmatrix}.$$
Let
			\[W=
			\begin{pmatrix} 
			1 & V\\ 
			0 & 1
			\end{pmatrix}
			\begin{pmatrix} 
			1 & 0\\ 
			-U & 1
			\end{pmatrix}
			\begin{pmatrix} 
			1 & V\\
			0 & 1
			\end{pmatrix}
			\begin{pmatrix}
			0 & -1\\
			1 & 0
			\end{pmatrix},
			\quad e_{1,1}=\begin{pmatrix}
			1&0\\0&0
			\end{pmatrix},
			\]
			and
			\begin{align*}
			P_\chi(D)&=We_{1,1}W^{-1}=\begin{pmatrix}1-(1-VU)^2 & (2-VU)V(1-UV)\\U(1-VU)&(1-UV)^2\end{pmatrix}.
			\end{align*}
			Then $P_\chi(D)$ and $e_{1,1}$ belong to $\textnormal{Idem}^{\varepsilon,r,N}_2(C^*(M)^+)$; see for example \cite[section 2.8]{WillettYu}.	
	The \emph{$(\varepsilon,r,N)$-quantitative higher index} of $D$ is the element
\begin{equation}
\label{eq index}
\Ind^{\varepsilon,r,N}(D)=\left[P_\chi(D)\right]-
			\left[e_{1,1}\right]\in K_0^{\varepsilon,r,N}(C^*(M)).
\end{equation}
\end{definition}
\subsection{Pairing quantitative and Lipschitz $K$-theories}
We now define a pairing between even quantitative $K$-theory of the Roe algebra and Lipschitz topological $K$-theory of $M$ taking values in $K_0(\cK)\cong\Z$:
$$\langle\cdot,\cdot\rangle_{\varepsilon,r,N,L}\colon K_0^{\varepsilon,r,N}(C^*(M))\times K_0^L(C_0(M))\to K_0(\cK)$$
that is well-defined under suitable conditions on the parameters $\varepsilon, r, N,$ and $L$ (see Theorem \ref{thm pairing defined}). This is a quantitative analogue of the natural pairing between $K$-homology and topological $K$-theory, and indeed factors through the homomorphism
$$\kappa\colon K_0^{\varepsilon',r',N'}(\cK)\to K_0(\cK)$$
given in $\eqref{eq kappa}$, where the parameters $\varepsilon'$, $r'$, and $N'$ are specified in Theorem \ref{thm pairing defined}.
\begin{remark}
Although our exposition focuses on the case of Riemannian spin manifolds, this pairing makes sense for metric spaces.	
\end{remark}

To give the basic idea behind this pairing, let us write it down in its most rudimentary form. Suppose $\alpha\in K_0^{\varepsilon,r,N}(C^*(M))$ is represented by a single $(\varepsilon, r,N)$-quasiidempotent $T\in C^*(M)$ and $\beta\in K_0^L(C_0(M))$ is represented by a single $L$-Lipschitz projection $p\in M_n(C_0(M))$. Let $T_n=T\otimes I_n$ denote the $n$-fold amplification of $T$. Then local compactness of $T$ implies that $T_n p\in M_n(\mathcal{K})$, while it can be shown that, by taking $L$ sufficiently small, the norm of the commutator $[T_n,p]$ can be made arbitrarily small, for fixed $\varepsilon$, $r$, and $N$. It follows that $T_n p$ is an $(\varepsilon',r,N')$-quasiidempotent for some new parameters $\varepsilon'$ and $N'$ (see Corollary \ref{cor basic}). The pairing then sends
\begin{equation}
\label{eq basic pairing}	
(\alpha,\beta)\mapsto[T_n p]\mapsto\kappa\big([T_np]\big),
\end{equation}
where $\kappa$ is as in \eqref{eq kappa}.

To treat the general case where $\alpha$ and $\beta$ are formal differences of matrices, we will need two ingredients. The first is the following estimate for commutators of amplified operators and matrix-valued functions. This is a generalization of \cite[Lemma 6.1.2]{WillettYu}.
\begin{proposition}
\label{prop commutator}
	Let $M$ and $S$ be as above. Let $T\in M_m(\mathcal{B}(L^2(S)))$ for some $m$ be an element with finite propagation. Let $f\colon M\to M_n(\mathbb{C})$ be a uniformly bounded $L$-Lipschitz map with respect to the operator norm on $M_n(\C)$. Then
	$$\norm{[T\otimes I_n,I_m\otimes f]}\leq 8L\cdot\textnormal{prop}(T)\norm{T},$$
	where $\otimes$ denotes the Kronecker product of matrices.
\end{proposition}
\begin{remark}
Recall that the Kronecker product of an $m\times n$ matrix $A$ (with entries $a_{ij}$) and a $p\times q$ matrix $B$ is the $pm\times qn$ matrix
$$A\otimes B=\begin{pmatrix}a_{11}B & \dots & a_{1n} B\\\vdots&\ddots&\vdots\\a_{m1}B & \dots & a_{mn}B\end{pmatrix}.$$	
\end{remark}


\begin{proofof}{Proposition \ref{prop commutator}}
	This is a multi-dimensional version of the proof of \cite[Lemma 6.1.2]{WillettYu}. For convenience, set $\delta=L\cdot\textnormal{prop}(T).$ By considering the real and imaginary parts of $f$ separately, it suffices to prove that for $f\colon M\to M_n(\R)$ a bounded continuous $L$-Lipschitz map we have
		$$\norm{[T\otimes I_n,I_m\otimes f]}\leq 4\delta\norm{T}.$$

Let $f_{i,j}\colon M\to\R$ denote the various matrix components of the function $f$, for $i,j\in\{1,\ldots n\}$. For each $k\in\Z$, let
$$M_{i,j,k} = f_{i,j}^{-1}\left[k\delta n^{-1},(k+1)\delta n^{-1}\right),$$
and let $\chi_{i,j,k}$ denote the characteristic function of $M_{i,j,k}$. Then for every $i$ and $j$, we can write $M$ as a disjoint union of $\{M_{i,j,k}\}_{k\in\Z}$. Define
		$$g_{i,j}=\sum_{k\in\Z}k\cdot\frac{\delta}{n}\chi_{i,j,k}.$$
		The $n^2$ functions $g_{i,j}$ form the components of a uniformly bounded Borel function $g\colon M\to M_n(\R)$ approximating $f$. We claim that:
	\begin{enumerate}[(i)]
	\item $\norm{I_m\otimes(f-g)}\leq\delta;$
	\item $\norm{[T\otimes I_n,I_m\otimes g]}\leq 2\delta\norm{T}.$
	\end{enumerate}
	The result then follows from this together with the observation that	
	$$\norm{[T\otimes I_n,I_m\otimes f]}\leq\norm{[T\otimes I_n,I_m\otimes(f-g)]}+\norm{[T\otimes I_n,I_m\otimes g]}.$$
	
	For (i), note that by construction, $\norm{f_{i,j}-g_{i,j}}_\infty\leq\frac{\delta}{n}$ for each $i$ and $j$. A standard inequality for matrix norms then implies that
		\begin{align*}
		\norm{I_m\otimes(f-g)}&\leq n\cdot\max_{i,j}\norm{f_{i,j}-g_{i,j}}_\infty\leq\delta.
		\end{align*}
	
	To prove (ii), first observe that if $k$ and $l$ are integers satisfying $|k-l|>1$, then for all $x\in M_k$ and $y\in M_l$ we have
	$$|f_{i,j}(x)-f_{i,j}(y)|>\frac{\delta}{n},$$
	for every $i,j\in\{1,\ldots,n\}$. Hence
	\begin{align*}
		\norm{f(x)-f(y)}&\geq\max_{i}\sum_{j=1}^n|f_{i,j}(x)-f_{i,j}(y)|>\delta,
	\end{align*}
	whence the definition of $\delta$ implies that $d_M(x,y)>\textnormal{prop}(T)$. It follows that
	\begin{align*}
	[T,I_m\otimes g_{i,j}]=&\sum_{k\in\Z}\frac{\delta}{n}(T\chi_{i,j,k}-\chi_{i,j,k}T)\cdot I_m\otimes k\\
	=&\begin{multlined}[t]
	\sum_{k\in\Z}\frac{\delta}{n}\Big((\chi_{i,j,k-1}+\chi_{i,j,k}+\chi_{i,j,k+1})T\chi_{i,j,k}\\-\chi_{i,j,k}T(\chi_{i,j,k-1}+\chi_{i,j,k}+\chi_{i,j,k+1})\Big)\cdot I_m\otimes k.
    \end{multlined}
    \end{align*}
    Rearranging and simplifying gives
    \begin{equation}
    \label{eq commutator}
   	[T,I_m\otimes g_{i,j}]=-\frac{\delta}{n}\Big(\sum_{k\in\Z}\chi_{i,j,k}T\chi_{i,j,k+1}+\sum_{k\in\Z}\chi_{i,j,k+1}T\chi_{i,j,k}\Big)\cdot I_m.
    \end{equation}
    The summands of each sum are pairwise orthogonal with norms bounded by $\norm{T}$, so the norm of the expression \eqref{eq commutator} is at most $2\frac{\delta}{n}\norm{T}$. For each $i,j\in\{1,\ldots,n\}$, let $E^{i,j}$ be the $(n\times n)$-matrix with $(i,j)$-th entry equal to $1$ and all other entries equal to $0$. Then we can write
    $$[T\otimes I_n,I_m\otimes g]=\sum_{i,j}\,[T,I_m\otimes g_{i,j}]\otimes E^{i,j}.$$
    It follows that
    \begin{align*}
		\norm{[T\otimes I_n,I_m\otimes g]}&\leq n\cdot\max_{i,j}\norm{[T,I_m\otimes g_{i,j}]}\leq 2\delta\norm{T}.
	\end{align*}
	This establishes (ii) and hence the result.
\end{proofof}

We introduce the following notation: for a given subset $W\subseteq\mathcal{B}(L^2(S))$, let us write $\langle W\rangle$ for the $C^*$-subalgebra generated by $W$. 
\begin{corollary}
\label{cor basic}
Let $P\in\textnormal{Idem}^{\varepsilon,r,N}_m(C^*(M)^+)$ and $p\in P^L(M_n(C_0(M)^+))$. Then we have
$$(P\otimes I_n)\cdot (I_m\otimes p)\in\textnormal{Idem}^{8rN^2L+\varepsilon,r,N}(M_{nm}\langle C^*(M)\cup C_0(M)\rangle^+).$$
\end{corollary}
\begin{proof}
Write $((P\otimes I_n)\cdot (I_m\otimes p))^2-(P\otimes I_n)\cdot (I_m\otimes p)$ as
\begin{align*}
	((P^2-P)\otimes I_n)\cdot (I_m\otimes p)+(P\otimes I_n)[P\otimes I_n,I_m\otimes p](I_m\otimes p)
\end{align*}
and take norms, applying Proposition \ref{prop commutator} to the term $[P\otimes I_n,I_m\otimes p]$.
\end{proof}


In order to extend the basic pairing \eqref{eq basic pairing} to formal differences of matrices, we will use a version of the \emph{difference construction} from \cite{KasparovYu}; we refer to that paper for more details (see also \cite[section 3]{WXYdecay}).

Let $A$ be unital a geometric $C^*$-algebra and $I$ a closed two-sided ideal. We will assume that $I$ is a ``geometric ideal" in the sense that the geometric structure of $A$ induces a geometric structure on $I$. Suppose $\alpha,\beta\in\textnormal{Idem}^{\varepsilon,r,N}(A)$ with $\alpha-\beta\in I$. Write
$$Z_\beta=\begin{pmatrix}\beta&0&1-\beta&0\\1-\beta&0&0&\beta\\0&0&\beta&1-\beta\\0&1&0&0\end{pmatrix},\quad Y_{\alpha,\beta}=\begin{pmatrix}\alpha&0&0&0\\0&1-\beta&0&0\\0&0&0&0\\0&0&0&0\end{pmatrix}.$$
Define
\begin{align}
\label{eq difference matrix}
	d(\alpha,\beta)&=Z_\beta^T Y_{\alpha,\beta}Z_\beta\nonumber\\
	&=\begin{pmatrix}1+\beta(\alpha-\beta)\beta&0&\beta(\alpha-\beta)&0\\0&0&0&0\\(\alpha-\beta)\alpha\beta&0&(1-\beta)(\alpha-\beta)(1-\beta)&0\\0&0&0&0\end{pmatrix}\in M_4(I^+).
\end{align}
\begin{proposition}
\label{prop difference}
Let $A$ and $I$ be as above. Let $\alpha$ and $\beta$ be $(\varepsilon,r,N)$-quasiidempotents such that $\alpha-\beta\in I$. Then $d(\alpha,\beta)$ is a $(2^8N^4\varepsilon,3r,16N^3)$-quasiidempotent in $M_4(I^+)$.
\end{proposition}
\begin{proof}
First observe that 
$$Z_\beta^T Z_\beta=
\begin{pmatrix}
1+2(\beta^2-\beta)&0&\beta-\beta^2&0\\
\beta-\beta^2&1+2(\beta^2-\beta)&\beta-\beta^2&0\\
\beta-\beta^2&\beta-\beta^2&1+2(\beta^2-\beta)&0\\
0&0&0&1	
\end{pmatrix}.$$
Recall that the operator norm of an $n\times n$ matrix $M$ with entries $m_{i,j}$ is bounded above by $n\cdot\max\{|m_{i,j}|\}$. Combined with $\norm{\beta-\beta^2}<\varepsilon$ and $\norm{1+2(\beta^2-\beta)-1}<2\varepsilon$, this implies that
 $\norm{Z_\beta^T Z_\beta-I_4}<8\varepsilon$, $\norm{Z_\beta}\leq 4N$, $\norm{Y_{\alpha,\beta}}\leq N$, and $\norm{Y_{\alpha,\beta}^2-Y_{\alpha,\beta}}\leq\varepsilon$. 
It follows that 
\begin{align*}
\norm{d(\alpha,\beta)^2-d(\alpha,\beta)}
&=\norm{Z_\beta^T Y_{\alpha,\beta}(Z_\beta Z_\beta^T)Y_{\alpha,\beta}Z_\beta-Z_\beta^T Y_{\alpha,\beta}Z_\beta\\&\qquad\qquad\qquad+Z_\beta^T Y_{\alpha,\beta}Y_{\alpha,\beta}Z_\beta-Z_\beta^T Y_{\alpha,\beta}Y_{\alpha,\beta}Z_\beta}\\
&=\norm{Z_\beta^T Y_{\alpha,\beta}(Z_\beta Z_\beta^T-I_4)Y_{\alpha,\beta}Z_\beta}\\
&\qquad\qquad\qquad\qquad+\norm{Z_\beta^T}\norm{Y_{\alpha,\beta}^2-Y_{\alpha,\beta}}\norm{Z_\beta}\\
&\leq 4N\cdot N\cdot 8\varepsilon\cdot N\cdot 4N+4N\cdot\varepsilon\cdot 4N\\
&\leq2^8N^4\varepsilon.
\end{align*}
Observing that $\norm{I_4-Y_{\alpha,\beta}}\leq N$, we see that $\norm{d(\alpha,\beta)}$ and $\norm{1-d(\alpha,\beta)}$ are each bounded by $4N\cdot N\cdot 4N=16N^3$. Finally, it is clear from \eqref{eq difference matrix} that $\textnormal{prop}(d(\alpha,\beta))\leq 3r$.
\end{proof}

Having made these preparations, we are now ready to formulate a general pairing between the quantitative $K$-theory of $C^*(M)$ and Lipschitz $K$-theory of $C_0(M)$. The idea of the pairing is to apply the difference construction twice, first to the ideal
$$\langle\cK\cup C_0(M)\rangle\triangleleft\langle C_0(M)\cup C^*(M)\rangle^+,$$
and then to the ideal
$$\cK\triangleleft\langle\cK\cup C_0(M)\rangle^+.$$

Suppose we have classes
\begin{align*}
[P_1]-[P_2]&\in K_0^{\varepsilon,r,N}(C^*(M)),\\
[p_1]-[p_2]&\in K_0^L(C_0(M))
\end{align*}
given by formal differences of elements $P_1,P_2\in\textnormal{Idem}_m^{\varepsilon,r,N}(C^*(M)^+)$ and $p_1,p_2\in P_n^L(C_0(M)^+)$ for some $m,n\in\N$. Suppose in addition that $p_1-p_2\in M_n(C_0(M))$; we may always arrange for this to be the case. For each $i,j=1,2$, Corollary \ref{cor basic} implies that 
$$(P_i\otimes I_n)\cdot(I_m\otimes p_j)$$ 
is an element of $\textnormal{Idem}_{mn}^{8rN^2L+\varepsilon,r,N}(\langle C_0(M)\cup C^*(M)\rangle^+)$. Let
\begin{equation*}
\label{eq first difference}
P_{i,p_1,p_2}=d\big((P_i\otimes I_n)\cdot(I_m\otimes p_1)\,,\,(P_i\otimes I_n)\cdot(I_m\otimes p_2)\big),
\end{equation*}
for $i=1,2$. It follows from Proposition \ref{prop difference} that $P_{i,p_1,p_2}$ is an element of $\textnormal{Idem}_{4mn}^{2^{11}rN^6L+2^8N^4\varepsilon,3r,16N^3}(\langle\cK\cup C_0(M)\rangle^+)$. Define
\begin{equation}
\label{eq pairing}
\langle [P_1]-[P_2],[p_1]-[p_2]\rangle=\big[d(P_{1,p_1,p_2},P_{2,p_1,p_2})\big]-\left[\begin{psmallmatrix}1&0&0&0\\0&0&0&0\\0&0&0&0\\0&0&0&0\end{psmallmatrix}\otimes I_{4mn}\right].
\end{equation}
It follows from Proposition \ref{prop difference} that $d(P_{1,p_1,p_2},P_{2,p_1,p_2})$ is an element of $\textnormal{Idem}_{16mn}^{2^{35}rN^{18}L+2^{32}N^{16}\varepsilon,9r,2^{16}N^9}(\cK^+)$.

\begin{theorem}
\label{thm pairing defined}
The formula \eqref{eq pairing} defines a map
\begin{equation}
\label{eq precise pairing}	
\langle\cdot,\cdot\rangle\colon K_0^{\varepsilon,r,N}(C^*(M))\times K_0^L(C_0(M))\to K_0^{\varepsilon',r',N'}(\cK),
\end{equation}
where $\varepsilon' = 2^{70}rN^{18}L + 2^{64}N^{16}\varepsilon$, $r'=9r$,and $N'=2^{32}N^9.$
\end{theorem}
\begin{proof}
To see that the first map is well-defined, suppose that for $i=1,2$, we have $(\varepsilon,r,N)$-quasiidempotents $P_i$ and $P_i'$ and $L$-Lipschitz projections $p_i$ and $p_i'$ such that there exist paths $P_{i}(t)$ and $p_{i}(t)$ of $(4\varepsilon,r,4N)$-quasiidempotents and $2L$-Lipschitz projections respectively with $P_{i}(0)=P_i$, $P_{i}(1)=P_i'$, $p_{i}(0)=p_1$, and $p_{i}(1)=p_i'$. Let
$$P'_{i,p'_1,p'_2}=d\big((P'_i\otimes I_n)\cdot(I_m\otimes p'_1)\,,\,(P'_i\otimes I_n)\cdot(I_m\otimes p'_2)\big).$$
The homotopies $P_i(t)$ and $p_i(t)$ induce a homotopy connecting $P_{i,p_1,p_2}$ and $P'_{i,p_1',p_2'}$, which in turn induces a homotopy connecting $d(P_{1,p_1,p_2},P_{2,p_1,p_2})$ and $d(P'_{1,p'_1,p'_2},P'_{2,p'_1,p'_2})$. Applying a similar analysis as above, we find that the latter is a homotopy through $(2^{72}rN^{18}L+2^{66}N^{16}\varepsilon,9r,2^{34}N^9)$-quasiidempotents in $M_{16mn}(\cK^+)$. By Definition \ref{def quantitative K}, the elements defined by pairing $P_1-P_2$ with $p_1-p_2$, and by pairing $P_1'-P_2'$ with $p_1'-p_2'$, are equal in $K_0^{\varepsilon',r',N'}(\cK)$.
\end{proof}
\begin{remark}
\label{rem zero}
It follows from this, together with \eqref{eq pairing}, that the pairing of the zero class in $K_0^{\varepsilon,r,N}(C^*(M))$ with any class in $K_0^L(C_0(M))$ is the zero class in $K_0^{\varepsilon',r',N'}(\cK)$.
\end{remark}
\begin{definition}
\label{def pairable}
We say a quadruple $(\varepsilon,r,N,L)$ of positive real numbers is \emph{pairable} if $2^6rN^2L+\varepsilon<2^{-68}N^{-16}$.
\end{definition}
\begin{corollary}
\label{cor useful pairing}
If $(\varepsilon,r,N,L)$ is pairable, then we also have a well-defined pairing
$$\langle\cdot,\cdot\rangle_{\varepsilon,r,N,L}\colon K_0^{\varepsilon,r,N}(C^*(M))\times K_0^L(C_0(M))\to K_0(\cK)$$
given by composing the pairing \eqref{eq precise pairing} with the map $\kappa\colon K_0^{\varepsilon',r',N'}(\cK)\to K_0(\cK)$ in \eqref{eq kappa}, i.e. $\langle\alpha,\beta\rangle_{\varepsilon,r,N,L}=\kappa(\langle\alpha,\beta\rangle)$.	
\end{corollary}
\begin{proof}
The quadruple $(\varepsilon,r,N,L)$ is pairable if and only if $\varepsilon'<\frac{1}{4}$, in which case the map $\kappa$ is well-defined.	
\end{proof}

The pairing $\langle\cdot,\cdot\rangle_{\varepsilon,r,N,L}$ can be thought of as a quantitative version of the natural pairing
\begin{equation}
\label{eq usual K-homology pairing}
\langle\cdot,\cdot\rangle_0\colon K_0(M)\times K^0(M)\to\Z
\end{equation}
between $K$-homology and $K$-theory. One construction of the latter pairing is as follows. Recall that the \emph{localization algebra} of $M$ is the $C^*$-algebra formed by completing the $*$-algebra of functions $f\colon [0,\infty)\rightarrow C^*(M)$ that:
\begin{enumerate}[(i)]
\item are uniformly bounded and uniformly continuous;
\item satisfy $\textnormal{prop}(f(t))\rightarrow 0\quad\textnormal{as }t\rightarrow\infty,$
\end{enumerate}
	with respect to the norm $\|f\|\coloneqq\sup_t\|{f(t)}\|_{\mathcal{B}(L^2(S))}$.
	The $K$-homology of $M$ is isomorphic to $K_0(C^*_L(M))$ \cite{QiaoRoe,Yu3}.
	In this picture, the pairing \eqref{eq usual K-homology pairing} is given by the composition:
\begin{equation}
\label{eq pairing K-homology}	
\langle\cdot,\cdot\rangle_0\colon K_0(C^*_L(M))\times K_0(C_0(M))\to K_0(\mathcal{K}_\infty)\xrightarrow{\cong} K_0(\mathcal{K}),
\end{equation}
where $\cK_\infty$ is the quotient of the $C^*$-algebra of uniformly continuous bounded maps $[0,\infty)\to\cK$ by the ideal generated by those maps that vanish at infinity, while the marked isomorphism is induced by evaluation at $0$.

The pairings $\langle\cdot,\cdot\rangle_{\varepsilon,r,N,L}$ and $\langle\cdot,\cdot\rangle_0$ are compatible in the following sense:
\begin{proposition}
\label{prop pairing compatibility}
Suppose the quadruple $(\varepsilon,r,N,L)$ is pairable. Let $[f_1]-[f_2]\in K_0(C^*_L(M))$ for idempotents $f_1,f_2\in M_m(C^*_L(M)^+)$, and let $s$ be a non-negative real number such that
$$f_i(s)\in M_m(\mathbb{C}_r[M])\quad\textnormal{and}\quad\max\{\norm{f_i(s)},\norm{1-f_i(s)}\}\leq N$$
for $i=1,2$. Let $p_1$ and $p_2$ be idempotents in $P^L_n(M_n(C_0(M)^+))$ such that $p_1-p_2\in M_n(C_0(M))$. Then we have
$$\big\langle[f_1]-[f_2],[p_1]-[p_2]\big\rangle_0=\big\langle[f_1(s)]-[f_2(s)],[p_1]-,[p_2]\big\rangle_{\varepsilon,r,N,L}\in K_0(\cK).$$
In particular,
$$\big\langle[D],[p_1]-[p_2]\big\rangle_0=\big\langle\Ind^{\varepsilon,r,N}(D),[p_1]-,[p_2]\big\rangle_{\varepsilon,r,N,L},$$
where $D$ is the Dirac operator on $M$.
\end{proposition}
\begin{proof}
This follows by directly comparing the formula \eqref{eq pairing} for $\langle\cdot,\cdot\rangle_{\varepsilon,r,N,L}$ with the formula for $\langle\cdot,\cdot\rangle_0$ in the localization algebra picture (c.f. \cite[section 3.1]{WXYdecay}).	
\end{proof}


\section{Proof of the main theorem}
\label{sec proof general}
As an application of the pairing $\langle\cdot,\cdot\rangle_{\varepsilon,r,N,L}$, we will now prove Theorem \ref{thm main}. This uses the vanishing theorem for the quantitative higher index \cite[Theorem 1.1]{GXY3}, which is a quantitative analogue of Rosenberg's pioneering result that the higher index on the universal cover of a closed manifold is an obstruction to positive scalar curvature \cite{Rosenberg1}. Although \cite[Theorem 1.1]{GXY3} was proved for the \emph{maximal} version of the Roe algebra, it implies the following analogous result in the \emph{reduced} case:
\begin{theorem}
\label{thm vanishing}
There exists a universal constant $\omega_0$ such that the following holds. Let $M$ be an even-dimensional complete Riemannian spin manifold with Dirac operator $D$ and scalar curvature function $\kappa$. Let $0<\varepsilon<\frac{1}{20}$ and $N\geq 7$. For every $c>0$, if $\kappa\geq c$ uniformly on $M$, then
$$\Ind^{\varepsilon,r,N}(D)=0\in K_0^{\varepsilon,r,N}(C^*(M))$$
for all $r\geq\frac{\omega_0}{\sqrt{c}}$.
\end{theorem}
\vspace{0.1cm}
We can now finish the proof of the main theorem.
\vspace{0.2cm}

\begin{proofof}{Theorem \ref{thm main}}
Suppose that $m$ is even; the odd case can be proved by considering the manifold $M\times\R$ instead. First let us fix some constants. Throughout this proof, we will take $N=7$, $\omega_0$ as in Theorem \ref{thm vanishing}, and $L_{m-1}$ as in Theorem \ref{thm Lipschitz} (see also Remark \ref{rem Lk}).  We will prove the result with
$$k(R,m)\coloneqq C_0(mR^{-1}L_{m-1})^2,$$
where $C_0=2^{150}N^{36}\omega_0^2$. Let $r=\omega_0k(R,m)^{-1/2}$.

Pick a point $x_0\in M$ and $\delta>0$, and let $B_\delta(x_0)$ be the open ball around $x_0$ of radius $\delta$. Let $i\colon B_\delta(x_0)\hookrightarrow M$ be the inclusion and $i_!\colon C_0(B_\delta(x_0))\to C_0(M)$ the extension-by-zero homomorphism. Let $i_{!*}$ and $i_!^*$ be the induced maps on $K$-theory and $K$-homology respectively. Taking $\delta$ sufficiently small, let $\beta_\delta$ denote the Bott element in $K_0(C_0(B_\delta(x_0)))$, and define 
 $$\beta_M=i_{!*}\beta_\delta\in K_0(C_0(M)).$$ 
Let $\cU=\{U_i\}_{i\in\N}$ be a uniformly bounded good cover with Lebesgue number $R$ and $R$-multiplicity $m$. Let $N(\cU)$ denote its nerve complex. By Lemma \ref{lem Lipschitz phi}, the homotopy equivalence $\phi\colon M\to N(\cU)$ in \eqref{eq phi} defined using the partition of unity \eqref{eq varphi} is a proper $mR^{-1}$-Lipschitz map. 
Let $\phi^*$ be the isomorphism on $K$-theory induced functorially by $\phi$. Since $\mathcal{U}$ is a locally compact simplicial complex of dimension $m-1$, Theorem \ref{thm Lipschitz} implies that the class $(\phi^*)^{-1}\beta_M\in K_0(C_0(N(\cU)))$ can be represented by a difference of two elements $p_1,p_2\in P^{L_{m-1}}_\infty(C_0(N(\cU))^+)$, whence $\phi^*p_1$ and $\phi^*p_2$ are elements of $P^{mR^{-1}L_{m-1}}_\infty(C_0(M)^+)$ such that
$$[\phi^*p_1]-[\phi^*p_2]=\beta_M.$$ 

Now suppose $\kappa(x)>k(R,m)$ uniformly on $M$. Let $D$ and $D_{B_\delta(x_0)}$ denote the Dirac operators on $M$ and $B_\delta(x_0)$ respectively. One verifies directly that the quadruple $\big(\varepsilon,r,N,mR^{-1}L_{m-1}\big)$ is pairable for all $\varepsilon$ sufficiently small. Fix such an $\varepsilon$, and let $\chi$ be a normalizing function with $\norm{\chi}_\infty=1$ and $\textnormal{supp}\,\widehat{\chi}\subseteq[-\frac{r}{5},\frac{r}{5}]$. Then $\textnormal{Ind}^{\varepsilon,r,N}(D)=[P_\chi(D)]-[e_{1,1}]$, as in \eqref{eq index}. By Theorem \ref{thm vanishing}, $\textnormal{Ind}^{\varepsilon,r,N}(D)$ vanishes, hence
$$\big\langle\Ind^{\varepsilon,r,N}(D),[\phi^*p_1]-[\phi^*p_2]\big\rangle_{\varepsilon,r,N,mR^{-1}L_{m-1}}=0\in K_0(\cK)$$
(see Remark \ref{rem zero}). On the other hand, by Proposition \ref{prop pairing compatibility}, this is equal to
\begin{align*}
\label{eq non-vanishing pairing}
\big\langle[D],\beta_M\big\rangle_0&=\big\langle [D],i_{!*}\beta_{B_\delta}\big\rangle_0\\
&=\big\langle i_!^*[D],\beta_{B_\delta}\big\rangle_0\\
&=\big\langle [D_{B_\delta(x_0)}],\beta_{B_\delta}\big\rangle_0\neq 0,
\end{align*}
where we used that $i_!^*[D]=[D_{B_{\delta}(x_0)}]$ (see for example \cite[Lemma 9.6.9]{WillettYu}). This is
a contradiction.	
\end{proofof}

\bibliographystyle{plain}

\bibliography{../../../../BigBibliography/mybib}

\begin{thebibliography}{10}

\bibitem{Chung}
Yeong~Chyuan Chung.
\newblock Quantitative {$K$}-theory for {B}anach algebras.
\newblock {\em J. Funct. Anal.}, 274(1):278--340, 2018.

\bibitem{Gromov-Lawson2}
Mikhael Gromov and H.~Blaine Lawson, Jr.
\newblock Positive scalar curvature and the {D}irac operator on complete
  {R}iemannian manifolds.
\newblock {\em Inst. Hautes \'{E}tudes Sci. Publ. Math.}, (58):83--196 (1984),
  1983.

\bibitem{GromovLectures}
Misha Gromov.
\newblock Four lectures on scalar curvature.
\newblock In {\em Perspectives in Scalar Curvature (Eds. M. Gromov and H. B.
  Lawson)}. World Sci.

\bibitem{Guentner}
E~Guentner, R~Willett, and Guoliang Yu.
\newblock Dynamical complexity and controlled operator {K}-theory.
\newblock Ast\'erisque. To appear.

\bibitem{GXY3}
Hao Guo, Zhizhang Xie, and Guoliang Yu.
\newblock Quantitative {$K$}-theory, positive scalar curvature, and band width.
\newblock In {\em Perspectives in Scalar Curvature (Eds. M. Gromov and H. B.
  Lawson)}. World Sci.

\bibitem{Hatcher}
Allen Hatcher.
\newblock {\em Algebraic topology}.
\newblock Cambridge University Press, Cambridge, 2002.

\bibitem{KasparovYu}
Gennadi Kasparov and Guoliang Yu.
\newblock The coarse geometric {N}ovikov conjecture and uniform convexity.
\newblock {\em Adv. Math.}, 206(1):1--56, 2006.

\bibitem{Lawson-Michelsohn}
H.~Blaine Lawson, Jr. and Marie-Louise Michelsohn.
\newblock {\em Spin geometry}, volume~38 of {\em Princeton Mathematical
  Series}.
\newblock Princeton University Press, Princeton, NJ, 1989.

\bibitem{Oyono2}
Herv\'{e} Oyono-Oyono and Guoliang Yu.
\newblock On quantitative operator {$K$}-theory.
\newblock {\em Ann. Inst. Fourier (Grenoble)}, 65(2):605--674, 2015.

\bibitem{QiaoRoe}
Yu~Qiao and John Roe.
\newblock On the localization algebra of {G}uoliang {Y}u.
\newblock {\em Forum Math.}, 22(4):657--665, 2010.

\bibitem{Rosenberg1}
Jonathan Rosenberg.
\newblock {$C^{\ast} $}-algebras, positive scalar curvature, and the {N}ovikov
  conjecture.
\newblock {\em Inst. Hautes \'{E}tudes Sci. Publ. Math.}, (58):197--212 (1984),
  1983.

\bibitem{WXYdecay}
Jinmin {Wang}, Zhizhang {Xie}, and Guoliang {Yu}.
\newblock Decay of scalar curvature on uniformly contractible manifolds with
  finite asymptotic dimension.
\newblock {\em Comm. Pure Appl. Math.}
\newblock To appear.

\bibitem{WillettYu}
Rufus Willett and Guoliang Yu.
\newblock {\em Higher Index Theory}.
\newblock Cambridge Studies in Advanced Mathematics. Cambridge University
  Press, 2020.

\bibitem{Yu3}
Guoliang Yu.
\newblock Localization algebras and the coarse {B}aum-{C}onnes conjecture.
\newblock {\em $K$-Theory}, 11(4):307--318, 1997.

\bibitem{Yu1}
Guoliang Yu.
\newblock The {N}ovikov conjecture for groups with finite asymptotic dimension.
\newblock {\em Ann. of Math. (2)}, 147(2):325--355, 1998.

\end{thebibliography}

\end{document}